\documentclass[11pt,twoside]{article}
\usepackage{amsmath, amsthm, amscd, amsfonts, amssymb, graphicx, color}

\date{}

\begin{document}

\centerline{}

\centerline {\Large{\bf  Construction of frame relative to \,$n$-Hilbert space }}

\newcommand{\mvec}[1]{\mbox{\bfseries\itshape #1}}
\centerline{}
\centerline{\textbf{Prasenjit Ghosh}}
\centerline{Department of Pure Mathematics, University of Calcutta,}
\centerline{35, Ballygunge Circular Road, Kolkata, 700019, West Bengal, India}
\centerline{e-mail: prasenjitpuremath@gmail.com}
\centerline{}
\centerline{\textbf{T. K. Samanta}}
\centerline{Department of Mathematics, Uluberia College,}
\centerline{Uluberia, Howrah, 711315,  West Bengal, India}
\centerline{e-mail: mumpu$_{-}$tapas5@yahoo.co.in}

\newtheorem{Theorem}{\quad Theorem}[section]

\newtheorem{definition}[Theorem]{\quad Definition}

\newtheorem{theorem}[Theorem]{\quad Theorem}

\newtheorem{remark}[Theorem]{\quad Remark}

\newtheorem{corollary}[Theorem]{\quad Corollary}

\newtheorem{note}[Theorem]{\quad Note}

\newtheorem{lemma}[Theorem]{\quad Lemma}

\newtheorem{example}[Theorem]{\quad Example}

\newtheorem{result}[Theorem]{\quad Result}
\newtheorem{conclusion}[Theorem]{\quad Conclusion}

\newtheorem{proposition}[Theorem]{\quad Proposition}

\begin{abstract}
\textbf{\emph{In this paper, our aim is to introduce the concept of a frame in n-Hilbert space and describe some of their properties.\,We further discuss tight frame relative to n-Hilbert space.\,At the end,\,we study the relationship between frame and bounded linear operator in n-Hilbert space.}}
\end{abstract}
{\bf Keywords:} \emph{ n-inner product space,\;n-normed space, pseudo-inverse, frame,\\ \smallskip\hspace{2.5cm} tight frame.}


\section{Introduction}

\smallskip\hspace{.6 cm} In the study of vector spaces, one of the most fundamental concept is that of a basis.\;A basis provides us with an expansion of all vectors in terms of its elements.\;In infinite-dimensional Hilbert space, we are forced to work with infinite series and so depending on  the work on infinite series, different concepts of basis has been established which may contain infinitely many elements namely, Schauder basis, orthonormal basis etc.\;In fact, in a separable Hilbert space every element can be expressed as a infinite linear combination of an orthonormal basis.\;The condition linearly independentness is not being assumed to define such Schauder basis or orthonormal basis but Schauder basis or orthonormal basis automatically becomes linearly independent.\;A frame is also spanning set of a Hilbert space but it is a redundant or linearly dependent system for a Hilbert space.\;So, frame can be considered as a generalization of orthonormal basis.\;In fact, frames play important role in theoretical research of wavelet analysis, signal denoising, feature extraction, robust signal processing etc.\;In 1946, D. Gabor \cite{Gabor} first initiated a technique for rebuilding signals using a family of elementary signals.\;In 1952, Duffin and Schaeffer abstracted Gabor's method to define frame for Hilbert space in their fundamental paper \cite{Duffin}.\;Later on, frame theory was popularized by Daubechies, Grossman, Meyer \cite{Daubechies}.\;The concept of \,$2$-inner product space was first introduced by Diminnie, Gahler and White \cite{Diminnie} in 1970's.\;In 1989, A.\,Misiak \cite{Misiak} developed the generalization of a \,$2$-inner product space for \,$n \,\geq\, 2$.

In this paper, our focus is to study and characterize various properties of frame and tight frame relative to \,$n$-Hilbert space.\;Finally, we shall established that an image of a frame under a bounded linear operator will be a frame if and only if the operator is invertible and give a characterization of frame in terms of its pre-frame operator in \,$n$-Hilbert space. 

Throughout this paper,\;$H$\; will denote separable Hilbert space with inner product \,$\left<\,\cdot \,,\, \cdot\,\right>$\, and \,$l^{\,2}(\,\mathbb{N}\,)$\; denote the space of square summable scalar-valued sequences with index set of natural numbers \,$\mathbb{N}$.

\section{Preliminaries}

\begin{theorem}\cite{Christensen}\label{th1}
Let \,$H_{\,1},\, H_{\,2}$\; be two Hilbert spaces and \;$U \,:\, H_{\,1} \,\to\, H_{\,2}$\; be a bounded linear operator with closed range \;$\mathcal{R}_{\,U}$.\;Then there exists a bounded linear operator \,$U^{\dagger} \,:\, H_{\,2} \,\to\, H_{\,1}$\, such that \,$U\,U^{\dagger}\,x \,=\, x\; \;\forall\; x \,\in\, \mathcal{R}_{\,U}$.
\end{theorem}

\begin{note}
The operator \,$U^{\dagger}$\, defined in Theorem (\ref{th1}), is called the pseudo-inverse of \,$U$.
\end{note}

\begin{theorem}\cite{Kreyzig}\label{th1.051}
The set \,$\mathcal{S}\,(\,H\,)$\; of all self-adjoint operators on \,$H$\; is a partially ordered set with respect to the partial order \,$\leq$\, which is defined as for \,$T,\,S \,\in\, \mathcal{S}\,(\,H\,)$ 
\[T \,\leq\, S \,\Leftrightarrow\, \left<\,T\,f \,,\, f\,\right> \,\leq\, \left<\,S\,f \,,\, f\,\right>\; \;\forall\; f \,\in\, H.\] 
\end{theorem}

\begin{definition}\cite{Kreyzig}
A self-adjoint operator \,$U \,:\, H \,\to\, H$\, is called positive if \,$\left<\,U\,x \,,\,  x\,\right> \,\geq\, 0$\, for all \,$x \,\in\, H$.\;In notation, we can write \,$U \,\geq\, 0$.\;A self-adjoint operator \,$V \,:\, H \,\to\, H$\, is called a square root of \,$U$\, if \,$V^{\,2} \,=\, U$.\;If, in addition \,$V \,\geq\, 0$, then \,$V$\, is called positive square root of \,$U$\, and is denoted by \,$V \,=\, U^{\,\dfrac{1}{2}}$. 
\end{definition}

\begin{theorem}\cite{Kreyzig}\label{th1.05}
The positive square root \,$V \,:\, H \,\to\, H$\, of an arbitrary positive self-adjoint operator \,$U \,:\, H \,\to\, H$\, exists and is unique.\;Further, the operator \,$V$\, commutes with every bounded linear operator on \,$H$\, which commutes with \,$U$.
\end{theorem}

\begin{definition}\cite{Christensen}
A sequence \,$\left\{\,f_{\,i}\,\right\}_{i \,=\, 1}^{\infty} \,\subseteq\, H$\, is said to be a frame for \,$H$\, if there exist positive constants \,$A,\, B$\, such that
\[ A\; \|\,f\,\|^{\,2} \,\leq\, \;\sum\limits_{i \,=\, 1}^{\infty}\;  \left|\ \left <\,f \,,\, f_{\,i} \, \right >\,\right|^{\,2} \,\leq\, B \;\|\,f\,\|^{\,2}\; \;\forall\; f \,\in\, H.\]
The constants \,$A$\, and \,$B$\, are called frame bounds.\,If the collection \;$\left\{\,f_{\,i}\,\right\}_{i \,=\, 1}^{\infty}$\; satisfies
\[\sum\limits_{i \,=\, 1}^{\infty}\;  \left|\,\left <\,f \,,\, f_{\,i} \, \right >\,\right|^{\,2} \,\leq\, B \;\|\,f\,\|^{\,2}\; \;\forall\; f \,\in\, H \]
then it is called a Bessel sequence with bound \,$B$.
\end{definition}

\begin{theorem}\cite{Christensen}\label{th1.1}
Let \,$\left\{\,f_{\,i}\,\right\}_{i \,=\, 1}^{\infty}$\; be a sequence in \,$H$\; and \,$B \,>\, 0$\; be given.\;Then \,$\left\{\,f_{\,i}\,\right\}_{i \,=\, 1}^{\infty}$\, is a Bessel sequence with Bessel bound \,$B$\, if and only if the operator defined by \;$ T \,:\, l^{\,2}\,(\,\mathbb{N}\,) \,\to\, H,\; T\,\{\,c_{i}\,\} \,=\, \sum\limits_{i \,=\, 1}^{\infty} \;c_{\,i}\,f_{\,i}$\;
is bounded and \;$\|\,T\,\| \,\leq\, \sqrt{B}$.
\end{theorem}

\begin{definition}\cite{Christensen}
Let \,$\left\{\,f_{\,i}\,\right\}_{i \,=\, 1}^{\infty}$\; be a frame for \,$H$.\;Then the bounded linear operator \,$ T \,:\, l^{\,2}\,(\,\mathbb{N}\,) \,\to\, H$, defined by \,$T\,\{\,c_{i}\,\} \,=\, \sum\limits_{i \,=\, 1}^{\infty} \;c_{\,i}\,f_{\,i}$, is called  pre-frame operator and its adjoint operator \,$T\,^{\,\ast} \,:\, H \,\to\, l^{\,2}\,(\,\mathbb{N}\,)$, given by \;$T^{\,\ast}\,f \,=\, \left \{\, \left <\,f \,,\, f_{i} \,\right >\,\right \}_{i \,=\, 1}^{\infty}$\; is called the analysis operator.\;The operator \,$S \,:\, H \,\to\, H$\; defined by \,$S\,f \,=\, T\,T^{\,\ast}\,f \,=\, \sum\limits^{\infty}_{i \,=\, 1}\; \left <\,f \,,\, f_{\,i} \, \right >\,f_{\,i}$\; is called the frame operator.
\end{definition}

The frame operator \,$S$\; is bounded, positive, self-adjoint and invertible \cite{Christensen}.

\begin{definition}\cite{Mashadi}
A \,$n$-norm on a linear space \,$X$\, (\,over the field \,$\mathbb{K}$\, of real or complex numbers\,) is a function
\[\left(\,x_{\,1} \,,\, x_{\,2} \,,\, \cdots \,,\, x_{\,n}\,\right) \,\longmapsto\, \left\|\,x_{\,1} \,,\, x_{\,2} \,,\, \cdots \,,\, x_{\,n}\,\right\|,\; x_{\,1},\, x_{\,2},\, \cdots,\, x_{\,n} \,\in\, X\]from \,$X^{\,n}$\, to the set \,$\mathbb{R}$\, of all real numbers such that for every \,$x_{\,1},\, x_{\,2},\, \cdots,\, x_{\,n} \,\in\, X$\, and \,$\alpha \,\in\, \mathbb{K}$,
\begin{itemize}
\item[(I)]\;\; $\left\|\,x_{\,1} \,,\, x_{\,2} \,,\, \cdots \,,\, x_{\,n}\,\right\| \,=\, 0$\; if and only if \,$x_{\,1},\, \cdots,\, x_{\,n}$\; are linearly dependent,
\item[(II)]\;\;\; $\left\|\,x_{\,1} \,,\, x_{\,2} \,,\, \cdots \,,\, x_{\,n}\,\right\|$\; is invariant under permutations of \,$x_{\,1},\, x_{\,2},\, \cdots,\, x_{\,n}$,
\item[(III)]\;\;\; $\left\|\,\alpha\,x_{\,1} \,,\, x_{\,2} \,,\, \cdots \,,\, x_{\,n}\,\right\| \,=\, |\,\alpha\,|\, \left\|\,x_{\,1} \,,\, x_{\,2} \,,\, \cdots \,,\, x_{\,n}\,\right\|$,
\item[(IV)]\;\; $\left\|\,x \,+\, y \,,\, x_{\,2} \,,\, \cdots \,,\, x_{\,n}\,\right\| \,\leq\, \left\|\,x \,,\, x_{\,2} \,,\, \cdots \,,\, x_{\,n}\,\right\| \,+\,  \left\|\,y \,,\, x_{\,2} \,,\, \cdots \,,\, x_{\,n}\,\right\|$.
\end{itemize}
A linear space \,$X$, together with a n-norm \,$\left\|\,\cdot \,,\, \cdots \,,\, \cdot \,\right\|$, is called a linear n-normed space. 
\end{definition}

\begin{definition}\cite{Misiak}
Let \,$n \,\in\, \mathbb{N}$\; and \,$X$\, be a linear space of dimension greater than or equal to \,$n$\; over the field \,$\mathbb{K}$, where \,$\mathbb{K}$\, is the real or complex numbers field.\;A n-inner product on \,$X$\, is a map 
\[\left(\,x \,,\, y \,,\, x_{\,2} \,,\, \cdots \,,\, x_{\,n}\,\right) \,\longmapsto\, \left<\,x \,,\, y \;|\; x_{\,2} \,,\, \cdots \,,\, x_{\,n} \,\right>,\; x,\, y,\, x_{\,2},\, \cdots,\, x_{\,n} \,\in\, X\]from \,$X^{n \,+\, 1}$\, to the set \,$\mathbb{K}$\, such that for every \,$x,\, y,\, x_{\,1},\, x_{\,2},\, \cdots,\, x_{\,n} \,\in\, X$\, and \,$\alpha \,\in\, \mathbb{K}$,
\begin{itemize}
\item[(I)]\;\; $\left<\,x_{\,1} \,,\, x_{\,1} \;|\; x_{\,2} \,,\, \cdots \,,\, x_{\,n} \,\right> \;\geq\;  0$\; and \;$\left<\,x_{\,1} \,,\, x_{\,1} \;|\; x_{\,2} \,,\, \cdots \,,\, x_{\,n} \,\right> \;=\;  0$\; if and only if \;$x_{\,1},\, x_{\,2},\, \cdots,\, x_{\,n}$\; are linearly dependent,
\item[(II)]\;\; $\left<\,x \,,\, y \;|\; x_{\,2} \,,\, \cdots \,,\, x_{\,n} \,\right> \;=\; \left<\,x \,,\, y \;|\; x_{\,i_{\,2}} \,,\, \cdots \,,\, x_{\,i_{\,n}} \,\right> $\; for every permutations \\$\left(\, i_{\,2} \,,\, \cdots \,,\, i_{\,n} \,\right)$\; of \;$\left(\, 2 \,,\, \cdots \,,\, n \,\right)$,
\item[(III)]\;\; $\left<\,x \,,\, y \;|\; x_{\,2} \,,\, \cdots \,,\, x_{\,n} \,\right> \;=\; \overline{\left<\,y \,,\, x \;|\; x_{\,2} \,,\, \cdots \,,\, x_{\,n} \,\right> }$,
\item[(IV)]\;\; $\left<\,\alpha\,x \,,\, y \;|\; x_{\,2} \,,\, \cdots \,,\, x_{\,n} \,\right> \;=\; \alpha \,\left<\,x \,,\, y \;|\; x_{\,2} \,,\, \cdots \,,\, x_{\,n} \,\right> $,
\item[(V)]\;\; $\left<\,x \,+\, y \,,\, z \;|\; x_{\,2} \,,\, \cdots \,,\, x_{\,n} \,\right> \;=\; \left<\,x \,,\, z \;|\; x_{\,2} \,,\, \cdots \,,\, x_{\,n} \,\right> \,+\,  \left<\,y \,,\, z \;|\; x_{\,2} \,,\, \cdots \,,\, x_{\,n} \,\right>$.
\end{itemize}
A linear space \,$X$\, together with n-inner product \,$\left<\,\cdot \,,\, \cdot \,|\, \cdot \,,\, \cdots \,,\, \cdot\,\right>$\, is called n-inner product space.
\end{definition}

\begin{theorem}\cite{Gunawan}
For \,$n$-inner product space \,$\left(\,X \,,\, \left<\,\cdot \,,\, \cdot \,|\, \cdot \,,\, \cdots \,,\, \cdot\,\right>\,\right)$, 
\[\left|\,\left<\,x \,,\, y \,|\, x_{\,2} \,,\,  \cdots \,,\, x_{\,n}\,\right>\,\right| \,\leq\, \left\|\,x \,,\, x_{\,2} \,,\, \cdots \,,\, x_{\,n}\,\right\|\, \left\|\,y \,,\, x_{\,2} \,,\, \cdots \,,\, x_{\,n}\,\right\|\]
hold for all \,$x,\, y,\, x_{\,2},\, \cdots,\, x_{\,n} \,\in\, X$, is called Cauchy-Schwarz inequality.
\end{theorem}

\begin{theorem}\cite{Misiak}
For every n-inner product space \,$\left(\,X \,,\, \left<\,\cdot \,,\, \cdot \;|\; \cdot \,,\, \cdots \,,\, \cdot\,\right> \,\right)$,
\[\left \|\,x_{\,1} \,,\, x_{\,2} \,,\, \cdots \,,\, x_{\,n}\,\right\| \,=\, \sqrt{\left <\,x_{\,1} \,,\, x_{\,1} \;|\; x_{\,2} \,,\,  \cdots \,,\, x_{\,n}\,\right>}\] defines a n-norm for which
\[ \left <\,x,\, y \,|\, x_{\,2},\, \cdots,\, x_{\,n}\,\right> \,=\,\dfrac{\,1}{\,4}\, \left(\,\|\,x \,+\, y,\, x_{\,2},\, \cdots,\, x_{\,n}\,\|^{\,2} \,-\, \|\,x \,-\, y,\, x_{\,2},\, \cdots,\, x_{\,n}\,\|^{\,2}\,\right),\;\&\] 
\[\|\,x \,+\, y \,,\, x_{\,2} \,,\, \cdots \,,\, x_{\,n}\,\|^{\,2} \,+\, \|\,x \,-\, y \,,\, x_{\,2} \,,\, \cdots \,,\, x_{\,n}\,\|^{\,2}\]
\[ \,=\, 2\, \left(\,\|\,x \,,\, x_{\,2} \,,\, \cdots \,,\, x_{\,n}\,\|^{\,2} \,+\, \|\,y \,,\, x_{\,2} \,,\, \cdots \,,\, x_{\,n}\,\|^{\,2} \,\right)\] 
hold for all \;$x,\, y,\, x_{\,1},\, x_{\,2},\, \cdots,\, x_{\,n} \,\in\, X$.
\end{theorem}

\begin{definition}\cite{Mashadi}
Let \,$\left(\,X \,,\, \left\|\,\cdot \,,\, \cdots \,,\, \cdot \,\right\|\,\right)$\; be a linear n-normed space.\;A sequence \,$\{\,x_{\,k}\,\}$\; in \,$X$\, is said to converge to some \,$x \,\in\, X$\; if 
\[\lim\limits_{k \to \infty}\,\left\|\,x_{\,k} \,-\, x \,,\, x_{\,2} \,,\, \cdots \,,\, x_{\,n} \,\right\| \,=\, 0\]
for every \,$ x_{\,2},\, \cdots,\, x_{\,n} \,\in\, X$\, and it is called a Cauchy sequence if 
\[\lim\limits_{l \,,\, k \,\to\, \infty}\,\left \|\,x_{l} \,-\, x_{\,k} \,,\, x_{\,2} \,,\, \cdots \,,\, x_{\,n}\,\right\| \,=\, 0\]
for every \,$ x_{\,2},\, \cdots,\, x_{\,n} \,\in\, X$.\;The space \,$X$\, is said to be complete if every Cauchy sequence in this space is convergent in \,$X$.\;A n-inner product space is called n-Hilbert space if it is complete with respect to its induce norm.\;2-Hilbert space \cite{Kim} is a particular case of n-Hilbert space for \,$n \,=\, 2$.
\end{definition}

\begin{definition}\cite{Sadeghi}
Let \,$\left(\,X \,,\, \left<\,\cdot \,,\, \cdot \,|\, \cdot\, \right > \,\right)$\; be a 2-Hilbert space and \,$\xi \,\in\, X $.\,A sequence \,$\left\{\,f_{\,i}\,\right\}^{\infty}_{i \,=\, 1} \,\subseteq\, X$\, is said to be a 2-frame associated to \,$\xi$\, if there exist positive constants \,$A,\, B$\, such that
\[A\; \left\|\,f \,,\, \xi\,\right\|^{\,2} \,\leq\, \sum\limits_{i \,=\, 1}^{\infty}\, \left|\,\left <\,f \,,\,  f_{\,i} \,|\, \xi\, \right>\,\right|^{\,2} \,\leq\, B\, \left\|\,f \,,\, \xi\,\right\|^{\,2}\;\; \;\forall\; f \,\in\, X.\] 
\end{definition}

\begin{theorem}\cite{Sadeghi}
Let \,$L_{\,\xi}$\; denote the \,$1$-dimensional linear subspace of \,$X$\, generated by a fixed \;$\xi \,\in\, X$.\;Let \;$M_{\,\xi}$\, be the algebraic complement of \,$L_{\,\xi}$.\;Define \,$\left <\,x \,,\, y\, \right >_{\,\xi} \,=\, \left <\,x \,,\, y \,|\, \xi\, \right >$\, on \,$X$.\;This semi-inner product induces an inner product on the quotient space \,$X \,/\, L_{\,\xi}$\, which is given by
\[ \left <\,x \,+\, L_{\,\xi} \,,\, y \,+\, L_{\,\xi}\, \right >_{\,\xi} \,=\, \left <\,x \,,\, y\,\right >_{\,\xi} \,=\, \left <\,x \,,\, y \,|\, \xi\, \right >\; \;\forall \;\; x,\, y \,\in\, X.\]
By identifying \,$ X \,/\, L_{\,\xi}$\, with \,$M_{\,\xi}$\, in an obvious way, we obtain an inner product on \,$M_{\,\xi}$.\;Define \,$\left \|\,x\, \right \|_{\,\xi} \,=\, \sqrt{\, \left <\,x \,,\, x\, \right >_{\,\xi}} \;\;\; \left(\,x \,\in\, M_{\,\xi}\,\right)$.\;Then \,$\left(\,M_{\,\xi} \,,\, \|\,\cdot\,\|_{\,\xi}\,\right)$\, is a norm space.\;Let \,$X_{\,\xi}$\, be the completion of the inner product space \,$M_{\,\xi}$.
\end{theorem}

\begin{theorem}\cite{Sadeghi}\label{th1.2}
Let \,$\left(\,X \,,\, \left<\,\cdot \,,\, \cdot \,|\, \cdot\, \right > \,\right)$\; be a 2-Hilbert space and \,$\xi \,\in\, X$.\;Then a sequence \,$\left\{\,f_{\,i}\,\right\}^{\,\infty}_{\,i \,=\, 1}$\; in \,$X$\; is a 2-frame associated to \,$\xi$\, with bounds \,$A \;\;\&\;\; B$\; if and only if it is a frame for the Hilbert space \,$X_{\,\xi}$\; with bounds \,$A \;\;\&\;\; B$.
\end{theorem}

\section{Frame and it's properties in $n$-Hilbert space}

\smallskip\hspace{.6 cm} In this section, we introduce the notion of frame relative to \,$n$-Hilbert space and discussed several properties of them in \,$n$-Hilbert space.

\begin{theorem}
Let \,$\left(\,X \,,\, \left<\,\cdot \,,\, \cdot \,|\, \cdot \,,\, \cdots \,,\, \cdot\,\right> \,\right)$\, be a n-inner product space and \,$x_{\,1},\, \cdots,\, x_{\,n}$\, are elements in \,$X$.\;Then  
\[\left\|\,x_{\,1} \,,\, x_{\,2} \,,\, \cdots \,,\, x_{\,n}\,\right\|\,=\, \sup\,\left\{\,\left|\,\left<\,x_{\,1} \,,\, y \,|\, x_{\,2} \,,\, \cdots \,,\, x_{\,n}\,\right>\,\right| \,:\, y \,\in\, X,\, \left\|\,y \,,\, x_{\,2} \,,\, \cdots \,,\, x_{\,n}\,\right\| \,=\, 1 \,\right\}.\]
\end{theorem}

\begin{proof}
Now,
\[\left\|\,x_{\,1} \,,\, x_{\,2} \,,\, \cdots \,,\, x_{\,n}\,\right\| \,=\, \left<\,x_{\,1} \,,\, \dfrac{x_{\,1}}{\left\|\,x_{\,1} \,,\, x_{\,2} \,,\, \cdots \,,\, x_{\,n}\,\right\|} \;|\; x_{\,2} \,,\, \cdots \,,\, x_{\,n}\,\right>\]
\[\hspace{2cm}\leq\, \sup\,\left\{\,\left|\,\left<\,x_{\,1} \,,\, y \;|\; x_{\,2} \,,\, \cdots \,,\, x_{\,n}\,\right>\,\right| \;:\; y \,\in\, X\; \;\text{and}\;\; \left\|\,y \,,\, x_{\,2} \,,\, \cdots \,,\, x_{\,n}\,\right\| \,=\, 1 \,\right\}\]
\[\hspace{3cm}\left[\;\text{where}\;\; y \,=\, \dfrac{x_{\,1}}{\left\|\,x_{\,1} \,,\, x_{\,2} \,,\, \cdots \,,\, x_{\,n}\,\right\|}\,\right]\]
\[\hspace{2cm}\leq\; \sup\,\left\{\,\left \|\,x_{\,1} \,,\, x_{\,2} \,,\, \cdots \,,\, x_{\,n}\,\right\|\,\left\|\,y \,,\, x_{\,2} \,,\, \cdots \,,\, x_{\,n}\,\right\| \;:\;  \left\|\,y \,,\, x_{\,2} \,,\, \cdots \,,\, x_{\,n}\,\right\| \,=\, 1 \,\right\}\]
\[\hspace{3cm}[\;\text{by Cauchy-Schwarz inequality}\;]\]
\[ =\; \left\|\,x_{\,1} \,,\, x_{\,2} \,,\, \cdots \,,\, x_{\,n}\,\right\|.\hspace{5.8cm} \]
\end{proof}

Let \,$L_{F}$\, denote the linear subspace of \,$X$\, spanned by the non-empty finite set \,$F \,=\, \left\{\,\,a_{\,2} \,,\, a_{\,3} \,,\, \cdots \,,\, a_{\,n}\,\right\}$, where \,$a_{\,2},\, a_{\,3},\, \cdots,\, a_{\,n}$\, are fixed elements in \,$X$.\;Then the quotient space \,$ X \,/\, L_{F}$\, is a normed linear space with respect to the norm, 
\[\left\|\,x \,+\, L_{F}\,\right\|_{F} \,=\, \left\|\,x \,,\, a_{\,2} \,,\,  \cdots \,,\, a_{\,n}\,\right\|\; \;\text{for every}\; x \,\in\, X.\]
Let \,$M_{F}$\, be the algebraic complement of \,$L_{F}$, then \,$X \,=\, L_{F} \,\oplus\, M_{F}$.\;Define   
\[\left<\,x \,,\, y\,\right>_{F} \,=\, \left<\,x \,,\, y \;|\; a_{\,2} \,,\,  \cdots \,,\, a_{\,n}\,\right>\; \;\text{on}\; \;X.\]
Then \,$\left<\,\cdot \,,\, \cdot\,\right>_{F}$\, is a semi-inner product on \,$X$\, and this semi-inner product induces an inner product on the quotient space \,$X \,/\, L_{F}$\; which is given by
\[\left<\,x \,+\, L_{F} \,,\, y \,+\, L_{F}\,\right>_{F} \,=\, \left<\,x \,,\, y\,\right>_{F} \,=\, \left<\,x \,,\, y \,|\, a_{\,2} \,,\,  \cdots \,,\, a_{\,n} \,\right>\;\; \;\forall \;\; x,\, y \,\in\, X.\]
By identifying \,$ X \,/\, L_{F}$\; with \,$M_{F}$\; in an obvious way, we obtain an inner product on \,$M_{F}$.\;Now, for every \,$x \,\in\, M_{F}$, we define \,$\|\,x\,\|_{F} \;=\; \sqrt{\left<\,x \,,\, x \,\right>_{F}}$\, and it can be easily verify that \,$\left(\,M_{F} \,,\, \|\,\cdot\,\|_{F}\,\right)$\; is a norm space.\;Let \,$X_{F}$\; be the completion of the inner product space \,$M_{F}$.\\

For the remaining part of this paper, \,$\left(\,X \,,\, \left<\,\cdot \,,\, \cdot \,|\, \cdot \,,\, \cdots \,,\, \cdot\,\right> \,\right)$\; is consider to be a \,$n$-Hilbert space.\;$I_{F}$\, will denote the identity operator on \,$X_{F}$\, and \,$\mathcal{B}\,(\,X_{F}\,)$\, denote the space of all bounded linear operator on \,$X_{F}$. 

\begin{definition}\label{def0.1}
Let \,$X$\, be a n-Hilbert space.\;A sequence \,$\left\{\,f_{\,i}\,\right\}^{\,\infty}_{\,i \,=\, 1}$\, in \,$X$\, is said to be a frame associated to \,$\left(\,a_{\,2},\, \cdots,\, a_{\,n}\,\right)$\, for \,$X$\, if there exist constant \,$0 \,<\, A \,\leq\, B \,<\, \infty$\, such that 
\[ A \, \left\|\,f \,,\, a_{\,2} \,,\, \cdots \,,\, a_{\,n} \,\right\|^{\,2} \,\leq\, \sum\limits^{\infty}_{i \,=\, 1}\,\left|\,\left<\,f \,,\, f_{\,i} \,|\, a_{\,2} \,,\, \cdots \,,\, a_{\,n}\,\right>\,\right|^{\,2} \,\leq\, B\, \left\|\,f \,,\, a_{\,2} \,,\, \cdots \,,\, a_{\,n}\,\right\|^{\,2}\]
for all \,$f \,\in\, X$.\;The infimum of all such \,$B$\, is called the optimal upper frame bound and supremum of all such \,$A$\, is called the optimal lower frame bound.\;A sequence \,$\left\{\,f_{\,i}\,\right\}^{\,\infty}_{\,i \,=\, 1}$\, satisfies the inequality 
\[\sum\limits^{\infty}_{i \,=\, 1}\,\left|\,\left<\,f \,,\, f_{\,i} \,|\, a_{\,2} \,,\, \cdots \,,\, a_{\,n}\,\right>\,\right|^{\,2} \,\leq\, B\; \left\|\,f \,,\, a_{\,2} \,,\, \cdots \,,\, a_{\,n}\,\right\|^{\,2}\; \;\forall\; f \,\in\, X\] is called a Bessel sequence associated to \,$\left(\,a_{\,2},\, \cdots,\, a_{\,n}\,\right)$\, in \,$X$\, with bound \,$B$.
\end{definition}

\begin{theorem}\label{th2}
Let \,$\left(\,X \,,\, \left<\,\cdot \,,\, \cdot \,|\, \cdot \,,\, \cdots \,,\, \cdot\,\right> \,\right)$\; be a n-Hilbert space.\,Then \,$\left\{\,f_{\,i}\,\right\}^{\,\infty}_{\,i \,=\, 1} \,\subseteq\, X$\; is a frame associated to \,$\left(\,a_{\,2},\, \cdots,\, a_{\,n}\,\right)$\; with bounds \,$A \;\;\&\;\; B$\; if and only if it is a frame for the Hilbert space \,$X_{F}$\; with bounds \,$A \;\;\&\;\; B$.
\end{theorem}

\begin{proof}
This theorem is an extension  of the Theorem (\ref{th1.2}) and proof of this theorem directly follows from that of the Theorem (\ref{th1.2}). 
\end{proof}

\begin{theorem}\label{th2.1}
Let \,$\left\{\,f_{\,i}\,\right\}^{\,\infty}_{\,i \,=\, 1}$\; be a Bessel sequence associated to \,$\left(\,a_{\,2},\, \cdots,\, a_{\,n}\,\right)$\; in \,$X$\, with bound \,$B$.\;Then the operator given by
\[T_{F} \,:\, l^{\,2}\,(\,\mathbb{N}\,) \,\to\, X_{F},\; T_{F}\,\left(\,\left\{\,c_{i}\,\right\}_{i \,=\, 1}^{\,\infty}\,\right) \,=\, \sum\limits^{\infty}_{i \,=\, 1}\;c_{\,i}\,f_{\,i}\] is well-defined and bounded.\;Furthermore, the adjoint operator of \,$T_{F}$\; is given by  
\[T_{F}^{\,\ast} \,:\, X_{F} \,\to\, l^{\,2}\,(\,\mathbb{N}\,),\; T_{F}^{\,\ast}\,(\,f\,) \;=\; \left\{\,\left<\,f \;,\; f_{\,i} \;|\; a_{\,2} \;,\; \cdots \;,\; a_{\,n} \,\right>\,\right\}_{i \,=\, 1}^{\,\infty}.\]
\end{theorem}

\begin{proof}
Let \,$\left\{\,c_{i}\,\right\}_{i \,=\, 1}^{\,\infty} \,\in\, l^{\,2}\,(\,\mathbb{N}\,)$.\;Then
\[\left\|\,\sum\limits^{l}_{i \,=\, 1}\,c_{\,i} \, f_{\,i} \,-\, \sum\limits^{k}_{i \,=\, 1}\,c_{i}\, f_{\,i}\,\right\|_{F}^{\,2} \,=\, \left\|\,\sum\limits^{l}_{i \,=\, k+1}\,c_{\,i}\,f_{\,i}\, \;,\; a_{\,2} \,,\, \cdots \,,\, a_{\,n}\,\right\|^{\,2}\hspace{3cm}\]
\[=\; \sup\,\left\{\,\left|\,\left<\,\sum\limits^{l}_{i \,=\, k+1}\,c_{\,i}\,f_{\,i}\, \;,\; y \,|\, a_{\,2} \,,\, \cdots \,,\, a_{\,n}\,\right>\,\right|^{\,2} \,:\, y \,\in\, X,\, \left\|\,y \,,\, a_{\,2} \,,\, \cdots \,,\, a_{\,n}\,\right\| \,=\, 1 \,\right\}\]
\[=\; \sup\,\left\{\,\left |\,\sum\limits^{l}_{i \,=\, k+1}\,c_{\,i}\,\left<\,f_{\,i}\, \;,\; y \;|\; a_{\,2} \;,\; \cdots \;,\; a_{\,n}\,\right>\,\right |^{\,2} \;:\; y \;\in\; X,\; \left \|\,y \;,\; a_{\,2} \;,\; \cdots \;,\; a_{\,n}\,\right\| \,=\, 1 \,\right\}\] 
\[\leq\, \sum\limits^{l}_{i \,=\, k+1}\,\left|\,c_{\,i}\,\right|^{\,2}\, \sup\,\left\{\,\sum\limits^{l}_{i \,=\, k+1}\,\left |\,\left<\,f_{\,i}\, \,,\, y \,|\, a_{\,2} \,,\, \cdots \,,\, a_{\,n}\,\right>\,\right|^{\,2} \,:\, y \,\in\, X,\, \left\|\,y \,,\, a_{\,2} \,,\, \cdots \,,\, a_{\,n}\,\right\| \,=\, 1 \,\right\}\]
\[\hspace{6cm}[\;\text{using Cauchy-Schwarz iequality}\;]\]
\[\leq\;  B\; \sum\limits^{l}_{i \,=\, k+1}\,|\,c_{\,i}\,|^{\,2}\;  [\;\; \text{since}\; \left\{\,f_{\,i}\,\right\}^{\,\infty}_{\,i \,=\, 1}\; \;\text{is a Bessel sequence associated to \,$\left(\,a_{\,2},\, \cdots,\, a_{\,n}\,\right)$}\;\;].\hspace{1.7cm}\]
This implies that \,$\sum\limits^{\infty}_{i \,=\, 1}\,c_{\,i}\,f_{\,i}$\; is convergent in \,$X_{F}$\; if \,$\{\,c_{\,i}\,\}^{\,\infty}_{\,i \,=\, 1} \;\in\; l^{\,2}\,(\,\mathbb{N}\,)$.\;Using the continuity of \,$n$-norm function, we get 
\[\left\|\,\sum\limits^{\infty}_{i \,=\, 1}\,c_{\,i}\,f_{\,i}\,\right\|_{F}^{\,2} \,\leq\, B\, \sum\limits^{\infty}_{i \,=\, 1}\,|\,c_{\,i}\,|^{\,2} \,\Rightarrow\, \left\|\,T_{F}\,\left\{\,c_{i}\,\right\}_{i \,=\, 1}^{\,\infty}\,\right\|_{F} \,\leq\, \sqrt{B}\, \left\|\,\left\{\,c_{i}\,\right\}_{i \,=\, 1}^{\,\infty}\,\right\|_{l^{\,2}}.\]The above calculation shows that \,$T_{F}$\; is well-defined and bounded.\;To find the expression for \,$T^{\,\ast}_{F}$, let \,$f \,\in\, X_{F}$\, and \,$\left\{\,c_{i}\,\right\}_{i \,=\, 1}^{\,\infty} \,\in\, l^{\,2}\,(\,\mathbb{N}\,)$.\;Then
\[\left<\,f \,,\, T_{F}\,\left\{\,c_{i}\,\right\}_{i \,=\, 1}^{\,\infty} \,|\, a_{\,2} \,,\, \cdots \,,\, a_{\,n} \,\right> \,=\, \left<\,f \,,\, \sum\limits^{\infty}_{i \,=\, 1}\,c_{\,i}\,f_{\,i} \,|\, a_{\,2} \,,\, \cdots \,,\, a_{\,n}\,\right>\hspace{1cm}\]
\[\hspace{4.7cm} \,=\, \sum\limits^{\infty}_{i \,=\, 1}\,\overline{\,c_{\,i}}\,\left<\,f \,,\, f_{\,i} \,|\, a_{\,2} \,,\, \cdots \,,\, a_{\,n} \,\right>.\]The convergence of the series \,$\sum\limits^{\infty}_{i \,=\, 1}\,\overline{\,c_{\,i}}\,\left<\,f \,,\, f_{\,i} \,|\, a_{\,2} \,,\, \cdots \,,\, a_{\,n} \,\right>$\, for all \,$\left\{\,c_{i}\,\right\}_{i \,=\, 1}^{\,\infty} \,\in\, l^{\,2}\,(\,\mathbb{N}\,)$\, implies that \,$\left\{\,\left<\,f \;,\; f_{\,i} \;|\; a_{\,2} \;,\; \cdots \;,\; a_{\,n} \,\right>\,\right\}_{i \,=\, 1}^{\,\infty} \,\in\, l^{\,2}\,(\,\mathbb{N}\,)$, (\,see  page 145 of \cite{Heuser}\,). 
\[\text{Therefore,}\hspace{.1cm}\left<\,f \,,\, T_{F}\,\left\{\,c_{i}\,\right\}_{i \,=\, 1}^{\,\infty}\,\right>_{F} \,=\, \left<\;\left\{\;\left<\,f \;,\; f_{\,i} \;|\; a_{\,2} \;,\; \cdots \;,\; a_{\,n} \,\right>\,\right\}_{i \,=\, 1}^{\,\infty} \;,\; \left\{\,c_{i}\,\right\}_{i \,=\, 1}^{\,\infty}\;\right>_{l^{\,2}\,(\,\mathbb{N}\,)}\]and hence \,$T_{F}^{\,\ast}\,(\,f\,) \;=\; \left\{\,\left<\,f \;,\; f_{\,i} \;|\; a_{\,2} \;,\; \cdots \;,\; a_{\,n} \,\right>\,\right\}_{i \,=\, 1}^{\,\infty}$.
\end{proof}

\begin{remark}
The operator \,$T_{F}$, defined in Theorem (\ref{th2.1}), is called pre-frame operator and the adjoint operator of \,$T_{F}$\; is called analysis operator for \,$\left\{\,f_{\,i}\,\right\}^{\,\infty}_{\,i \,=\, 1}$.
\end{remark}

\begin{definition}\label{def1}
Let \,$\left\{\,f_{\,i}\,\right\}^{\,\infty}_{\,i \,=\, 1}$\; be a frame associated to \,$\left(\,a_{\,2},\, \cdots,\, a_{\,n}\,\right)$\; for \,$X$. Then the operator \,$S_{F} \,:\, X_{F} \,\to\, X_{F}$\, defined by 
\[S_{F}\,(\,f\,) \,=\, \sum\limits^{\infty}_{i \,=\, 1}\, \left<\,f \,,\, f_{\,i} \,|\, a_{\,2} \,,\, \cdots \,,\, a_{\,n} \,\right>\,f_{\,i}\; \;\forall\; f \,\in\, X_{F}\]
is called the frame operator for \,$\left\{\,f_{\,i}\,\right\}^{\,\infty}_{\,i \,=\, 1}$.\;It can be easily verify that
\begin{equation}\label{eq1.1}
\left<\,S_{F}\,f \,,\, f \,|\, a_{\,2} \,,\, \cdots \,,\, a_{\,n}\,\right> \,=\, \sum\limits^{\infty}_{i \,=\, 1}\,\left|\,\left<\,f \,,\, f_{\,i} \,|\, a_{\,2} \,,\, \cdots \,,\, a_{\,n}\,\right>\,\right|^{\,2}\; \;\forall\; f \,\in\, X_{F}.
\end{equation}
\end{definition}

\begin{theorem}\label{th2.2}
Let \,$\left\{\,f_{\,i}\,\right\}^{\,\infty}_{\,i \,=\, 1}$\; be a frame associated to \,$\left(\,a_{\,2},\, \cdots,\, a_{\,n}\,\right)$\; for \,$X$\, with bounds \;$A \;\&\; B$.\;Then the corresponding frame operator \,$S_{F}$\; is bounded, invertible, self-adjoint and positive.
\end{theorem}

\begin{proof}
For each \,$f \,\in\, X_{F}$, we have 
\[\left\|\,S_{F}\,f \,\right\|_{F}^{\,2} \;=\; \left\|\,S_{F}\,f \;,\; a_{\,2} \;,\; \cdots \;,\; a_{\,n}\,\,\right\|^{\,2}\]
\[=\, \sup\,\left\{\,\left|\,\left<\,S_{F}\,f \,,\, g \;|\; a_{\,2} \,,\, \cdots \,,\, a_{\,n}\,\right>\,\right|^{\,2} \,:\, \left\|\,g \,,\, a_{\,2} \,,\, \cdots \,,\, a_{\,n}\,\right\| \,=\, 1 \,\right\}\hspace{3.5cm}\]
\[=\, \sup\,\left\{\,\left|\,\left<\,\sum\limits^{\infty}_{i \,=\, 1}\,\left<\,f \,,\, f_{\,i} \,|\, a_{\,2} \,,\, \cdots \,,\, a_{\,n} \,\right>\,f_{\,i} \;,\; g \;|\; a_{\,2} \,,\, \cdots \,,\, a_{\,n}\,\right>\,\right|^{\,2} \,:\,  \left\|\,g \,,\, a_{\,2} \,,\, \cdots \,,\, a_{\,n}\,\right\| \,=\, 1\,\right\}\]
\[\leq\, \sup\,\left\{\,\sum\limits^{\infty}_{i \,=\, 1}\, \left|\, \left<\,f \,,\, f_{\,i} \,|\, a_{\,2} \,,\, \cdots \,,\, a_{\,n} \,\right>\,\right|^{\,2}\,\sum\limits^{\infty}_{i \,=\, 1}\; \left|\,\left<\,g \,,\, f_{\,i} \,|\, a_{\,2} \,,\, \cdots \,,\, a_{\,n}\,\right>\,\right|^{\,2} \,:\, \left \|\,g \,,\, a_{\,2} \,,\, \cdots \,,\, a_{\,n}\,\right\| \,=\, 1 \,\right\}\]
\[\hspace{6cm}[\;\text{using Cauchy-Schwarz iequality}\;]\]
\[\leq\; B\; \left \|\,f \,,\, a_{\,2} \,,\, \cdots \,,\, a_{\,n}\,\right\|^{\,2}\;  B\; \;[\;\text{since}\; \left\{\,f_{\,i}\,\right\}^{\,\infty}_{\,i \,=\, 1}\; \;\text{is a frame associated to \,$\left(\,a_{\,2},\, \cdots,\, a_{\,n}\,\right)$}\;\;]\hspace{4cm}\]
\[ =\, B^{\,2}\; \left \|\,f \,,\, a_{\,2} \,,\, \cdots \,,\, a_{\,n}\,\right\|^{\,2}\;=\; B^{\,2}\; \|\,f\,\|^{\,2}_{F}.\hspace{6.9cm}\]
This shows that \,$S_{F}$\; is bounded.\;Since \,$S_{F} \,=\, T_{F}\,T^{\,\ast}_{F}$, it is easy to verify that \,$S_{F}$\, is self-adjoint.\;By (\ref{eq1.1}), frame inequality of definition (\ref{def0.1}), can be written as 
\[A\,\left<\,f \,,\, f \,|\, a_{\,2} \,,\, \cdots \,,\, a_{\,n}\,\right> \,\leq\, \left<\,S_{F}\,f \,,\, f \,|\, a_{\,2} \,,\, \cdots \,,\, a_{\,n}\,\right> \,\leq\, B\,\left<\,f \,,\, f \,|\, a_{\,2} \,,\, \cdots \,,\, a_{\,n}\,\right>\]
and therefore according to Theorem (\ref{th1.051}), we can write \,$A\,I_{F} \,\leq\, S_{F} \,\leq\, B\,I_{F}$.\;Thus, \,$S_{F}$\; is positive and consequently it is invertible.    
\end{proof}

\begin{remark}\label{rmr1}
In Theorem (\ref{th2.2}), it is proved that \,$A\,I_{F} \,\leq\, S_{F} \,\leq\, B\,I_{F}$.\;Since \,$S^{\,-1}_{F}$\; commutes with both \,$S_{F}$\; and \,$I_{F}$, multiplying in the inequality, \,$ A\,I_{F} \,\leq\, S_{F} \,\leq\, B\,I_{F}$\, by \,$S^{\,-1}_{F}$, we get \,$B^{\,-1}\,I_{F} \,\leq\, S^{\,-1}_{F} \,\leq\, A^{\,-1}\,I_{F}$.
\end{remark}

\begin{theorem}\label{th2.3}
Let \,$\{\,f_{\,i}\,\}^{\infty}_{i \,=1\, }$\; be a frame associated to \,$\left(\,a_{\,2},\, \cdots,\, a_{\,n}\,\right)$\; for \,$X$\, with frame bounds \,$A,\, B$\, and \,$S_{F}$\; be the corresponding frame operator.\;Then \,$\left \{\,S^{\,-1}_{F}\,f_{\,i}\,\right \}^{\infty}_{i \,=\, 1}$\; is also a frame associated to \,$\left(\,a_{\,2},\, \cdots,\, a_{\,n}\,\right)$\, with bounds \,$B^{\,-1},\, A^{\,-1}$. 
\end{theorem}

\begin{proof}
By Theorem (\ref{th2.2}), \,$S^{\,-1}_{F} \,:\, X_{F} \,\to\, X_{F}$\, is self-adjoint.\;Now, for each \,$f \,\in\, X_{F}$,
\[ \sum\limits^{\infty}_{i \,=\,1}\,\left |\,\left <\,f \,,\, S^{\,-1}_{F}\,f_{\,i} \,|\, a_{\,2} \,,\, \cdots \,,\, a_{\,n}\,\right >\,\right |^{\,2} \,=\, \sum\limits^{\infty}_{i \,=\, 1}\,\left|\,\left <\,\left (\,S^{\,-1}_{F}\,\right )^{\,\ast}\,f \,,\, f_{\,i} \,|\, a_{\,2} \,,\, \cdots \,,\, a_{\,n}\,\right >\,\right |^{\,2} \]
\[\;=\; \sum\limits^{\infty}_{i \,=\, 1}\,\left |\,\left <\,S^{\,-1}_{F}\; f \,,\, f_{\,i} \,|\, a_{\,2} \,,\, \cdots \,,\, a_{\,n}\,\right>\,\right|^{\,2}\]
\[ \,\leq\, B\; \left \|\,S^{\,-1}_{F}\, f \,,\, a_{\,2} \,,\, \cdots \,,\, a_{\,n}\right\|^{\,2}\hspace{1.2cm}\]
\[\hspace{3.5cm} [\;\text{since}\; \,\{\,f_{\,i}\,\}^{\infty}_{i \,=1\, } \;\; \text{is a frame associated to}\;\left(\,a_{\,2} \,,\, \cdots \,,\, a_{\,n}\,\right)\;]\]
\[\leq\, B\, \left \|\,S^{\,-1}_{F}\,\right \|^{\,2}\, \left \|\,f \,,\, a_{\,2} \,,\, \cdots \,,\, a_{\,n}\right \|^{\,2}.\]
This shows that \,$\left\{\,S^{\,-1}_{F}\;f_{\,i}\,\right \}^{\infty}_{i \,=\, 1}$\; is a Bessel sequence associated to \,$\left(\,a_{\,2},\, \cdots,\, a_{\,n}\,\right)$. Also, for any \,$f \,\in\, X_{F}$, we have
\[ \sum\limits^{\infty}_{i \,=\, 1}\,\left <\,f \,,\, S^{\,-1}_{F}\,f_{\,i} \,|\, a_{\,2} \,,\, \cdots \,,\, a_{\,n}\,\,\right >\, S^{\,-1}_{F}\, f_{\,i} \,=\, S^{\,-1}_{F}\,\,\left(\,\sum\limits^{\infty}_{i \,=\, 1}\,\left <\,S^{\,-1}_{F}\, f \,,\, f_{\,i} \,|\, a_{\,2} \,,\, \cdots \,,\, a_{\,n}\,\right >\,f_{\,i}\,\right) \]
\begin{equation}\label{eq2}
\hspace{6.5cm} \;=\; S^{\,-1}_{F}\; \left(\,S_{F}\; \left(\,S^{\,-1}_{F}\;f\,\right) \,\right) \;=\; S^{\,-1}_{F}\;\,f.
\end{equation}
This shows that \,$S^{\,-1}_{F}$\; is the frame operator for \,$\left\{\,S^{\,-1}_{F}\;f_{\,i}\,\right\}_{i \,=\, 1}^{\,\infty}$.\;Now, for each \,$f \,\in\, X_{F}$, using the inequality \,$B^{\,-1}\,I_{F} \,\leq\, S^{\,-1}_{F} \,\leq\, A^{\,-1}\,I_{F}$, we get
\[B^{\,-1}\,\|\,f \,,\, a_{\,2} \,,\, \cdots \,,\, a_{\,n}\,\,\|^{\,2} \,\leq\, \left <\,S^{\,-1}_{F}\,f \,,\, f \;|\;  a_{\,2} \,,\, \cdots \,,\, a_{\,n}\,\right > \,\leq\, A^{\,-1}\,\|\,f \,,\, a_{\,2} \,,\, \cdots \,,\, a_{\,n}\,\,\|^{\,2}.\]Now, using (\ref{eq2}),  
\,$\left <\,S^{\,-1}_{F}\;f \,,\, f \,|\,  a_{\,2} \,,\, \cdots \,,\, a_{\,n}\,\right >$ 
\[\,=\, \left <\,\sum\limits^{\infty}_{i \,=\, 1}\,\left <\,f \,,\, S^{\,-1}_{F}\;f_{\,i} \,|\,  a_{\,2} \,,\, \cdots \,,\, a_{\,n}\;\right >\,S^{\,-1}_{F}\;f_{\,i} \;,\; f \,|\, a_{\,2} \,,\, \cdots \,,\, a_{\,n}\, \right >\]
\[\;=\;\sum\limits^{\infty}_{i \,=\, 1}\,\left <\,f \,,\, S^{\,-1}_{F}\;f_{\,i} \,|\,  a_{\,2} \,,\, \cdots \,,\, a_{\,n}\;\right >\; \left <\,S^{\,-1}_{F}\;f_{\,i} \,,\, f \,|\,  a_{\,2} \,,\, \cdots \,,\, a_{\,n}\,\right >\hspace{.2cm}\]
\[\;=\; \sum\limits^{\infty}_{i \,=\, 1}\; \left|\,\left <\,f \,,\, S^{\,-1}_{F}\;f_{\,i} \,|\,  a_{\,2} \,,\, \cdots \,,\, a_{\,n}\,\right >\,\right|^{\,2}.\hspace{3.8cm}\]
Therefore for each \,$f \,\in\, X_{F}$,
\[ B^{\,-1}\,\|\,f \,,\, a_{\,2} \,,\, \cdots \,,\, a_{\,n}\,\,\|^{\,2} \,\leq\, \sum\limits^{\infty}_{i \,=\, 1}\, \left|\,\left <\,f \,,\, S^{\,-1}_{F}\, f_{\,i} \,|\,  a_{\,2} \,,\, \cdots \,,\, a_{\,n}\,\right >\,\right|^{\,2} \,\leq\, A^{\,-1}\,\|\,f \,,\, a_{\,2} \,,\, \cdots \,,\, a_{\,n}\,\,\|^{\,2}.\]
Hence, by Theorem (\ref{th2}), \,$\left\{\,S^{\,-1}_{F}\;f_{\,i}\,\right\}_{i \,=\, 1}^{\,\infty}$\; is a frame associated to \,$\left(\,a_{\,2},\, \cdots,\, a_{\,n}\,\right)$\; for \,$X$\, with bounds \,$B^{\,-1},\, A^{\,-1}$.                    
\end{proof}

\begin{remark}
From the Theorem (\ref{th2.3}), we can conclude that if \,$\left\{\,f_{\,i}\,\right\}^{\infty}_{i \,=\, 1}$\; is a frame associated to \,$\left(\,a_{\,2},\, \cdots,\, a_{\,n}\,\right)$\, with optimal frame bounds \,$A,\, B$, then \,$B^{\,-1},\, A^{\,-1}$\; are also optimal frame bounds for \,$\left\{\,S^{\,-1}_{F}\;f_{\,i}\,\right\}_{i \,=\, 1}^{\,\infty}$.\;The frame \,$\left\{\,S^{\,-1}_{F}\;f_{\,i}\,\right\}_{i \,=\, 1}^{\,\infty}$\; is called the canonical dual frame associated to \,$\left(\,a_{\,2},\, \cdots,\, a_{\,n}\,\right)$\, of \,$\left\{\,f_{\,i}\,\right\}^{\infty}_{i \,=\, 1}$.
\end{remark}

\begin{theorem}\label{eq2.4}
Let \,$\left\{\,f_{\,i}\,\right\}^{\infty}_{i \,=1\, }$\; be a frame associated to \,$\left(\,a_{\,2},\, \cdots,\, a_{\,n}\,\right)$\, for \,$X$\, and \,$S_{F}$\, be the corresponding frame operator.\;Then for every \,$f \,\in\, X_{F}$,
\[ f \;=\; \sum\limits^{\infty}_{i \,=\, 1}\,\left <\,f \;,\; S^{\,-1}_{F}\;f_{\,i} \;|\;  a_{\,2} \;,\; \cdots \;,\; a_{\,n}\;\right >\;f_{\,i}, \;\text{and}\] 
\[ f \;=\; \sum\limits^{\infty}_{i \,=\, 1}\,\left <\,f \;,\; f_{\,i} \;|\; a_{\,2} \;,\; \cdots \;,\; a_{\,n}\;\right >\;S^{\,-1}_{F}\;f_{\,i},\]
provided both the series converges unconditionally for all \,$f \,\in\, X_{F}$.
\end{theorem}

\begin{proof}
Let \,$f \,\in\, X_{F}$.\;Then 
\[f \,=\, S_{F}\, S^{\,-1}_{F}\,f \,=\, S_{F}\, \left(\,\sum\limits^{\infty}_{i \,=\, 1}\,\left <\,f \,,\, S^{\,-1}_{F}\,f_{\,i} \,|\, a_{\,2} \,,\, \cdots \,,\, a_{\,n}\,\,\right >\,S^{\,-1}_{F}\,f_{\,i}\,\right)\; [\;\text{using (\ref{eq2})}\;]\]
\[\hspace{.1cm} \;=\; \sum\limits^{\infty}_{i \,=\, 1}\,\left <\,f \,,\, S^{\,-1}_{F}\,f_{\,i} \,|\, a_{\,2} \,,\, \cdots \,,\, a_{\,n}\,\,\right >\,S_{F}\,\left(\,S^{\,-1}_{F}\,f_{\,i}\,\right)\]
\[\;=\; \sum\limits^{\infty}_{i \,=\, 1}\,\left <\,f \,,\, S^{\,-1}_{F}\,f_{\,i} \,|\, a_{\,2} \,,\, \cdots \,,\, a_{\,n}\,\,\right >\,f_{\,i}.\hspace{1.5cm}\]
Since \,$\left \{\,\left <\,f \,,\, S^{\,-1}_{F}\,f_{\,i} \,|\, a_{\,2} \,,\, \cdots \,,\, a_{\,n}\,\,\right >\,\right \}^{\infty}_{i \,=\, 1} \,\in\, l^{\,2} (\,\mathbb{N}\,)$\; and \,$\left\{\,f_{\,i}\,\right\}^{\infty}_{i \,=\, 1}$\; is a Bessel sequence associated to \,$\left(\,a_{\,2},\, \cdots,\, a_{\,n}\,\right)$, the above series converges unconditionally.\;On the other hand, 
\[f \,=\, S^{\,-1}_{F}\, S_{F}\, f \,=\, S^{\,-1}_{F}\, \left(\,\sum\limits^{\infty}_{i \,=\, 1}\,\left <\,f \,,\, f_{\,i} \,|\, a_{\,2} \,,\, \cdots \,,\, a_{\,n}\,\,\right >\,f_{\,i}\,\right)\hspace{1cm}\]
\[\hspace{3cm}\;=\; \sum\limits^{\infty}_{i \,=\, 1}\,\left <\,f \,,\, f_{\,i} \,|\,  a_{\,2} \,,\, \cdots \,,\, a_{\,n}\,\right >\;S^{\,-1}_{F} f_{\,i}\; \;\forall\; f \,\in\, X_{F}.\]
\end{proof}

\begin{definition}\label{def2}
A sequence \,$\left\{\,f_{\,i}\,\right\}^{\infty}_{i \,=1\, }$\, in \,$X$\, is said to be a tight frame associated to \,$\left(\,a_{\,2} ,\, \cdots,\, a_{\,n}\,\right)$\, for \,$X$\, with bound \,$A$\, if for all \,$f \,\in\, X$,
\begin{equation}\label{eq2.5}
\sum\limits^{\infty}_{i \,=\, 1}\,|\,\left<\,f \,,\, f_{\,i} \,|\, a_{\,2} \,,\, \cdots \,,\, a_{\,n}\,\right>\,|^{\,2} \,=\, A\,\left\|\,f \,,\, a_{\,2} \,,\, \cdots \,,\, a_{\,n}\,\right\|^{\,2}.
\end{equation}
If \,$A \,=\, 1$\, then it is called normalized tight frame associated to \,$\left(\,a_{\,2},\, \cdots,\, a_{\,n}\,\right)$. From (\ref{eq2.5}), we have
\[ \sum\limits^{\infty}_{i \,=\, 1}\,\left |\, \left <\, f \,,\,  \dfrac{1}{\,\sqrt{A}}\;f_{\,i} \,|\, a_{\,2} \,,\, \cdots \,,\, a_{\,n}\,\right>\,\right |^{\,2} \,=\, \left\|\,f \,,\, a_{\,2} \,,\, \cdots \,,\, a_{\,n}\,\right\|^{\,2}.\]
Therefore, if \,$\left\{\,f_{\,i}\,\right\}_{i \,=\,1}^{\infty}$\; is a tight frame associated to \,$\left(\,a_{\,2},\, \cdots,\, a_{\,n}\,\right)$\, with bound \,$A$\; then family \,$\left\{\,\dfrac{1}{\,\sqrt{A}}\;f_{\,i}\,\right\}^{\infty}_{i \,=\, 1}$\; is a normalized tight frame associated to \,$\left(\,a_{\,2},\, \cdots,\, a_{\,n}\,\right)$\,. According to Theorem (\ref{th2}), \,$\left\{\,f_{\,i}\,\right\}_{i \,=\,1}^{\infty}$\; is a tight frame associated to \,$\left(\,a_{\,2},\, \cdots,\, a_{\,n}\,\right)$\, for \,$X$\, with bound \,$A$\; if and only if it is a tight frame for \,$X_{F}$\, with bound \,$A$.                                                                                      
\end{definition}

\begin{theorem}
Let \,$\left\{\,f_{\,i}\,\right\}^{\infty}_{i \,=\, 1}$\, be a tight frame associated to \,$\left(\,a_{\,2},\, \cdots,\, a_{\,n}\,\right)$\, for \,$X$\, with frame bound \,$A$.\;Then for every \,$f \,\in\, X_{F}$,
\[ f \;=\; \dfrac{1}{A}\; \sum\limits^{\infty}_{i \,=\, 1}\, \left <\,f \,,\, f_{\,i} \,|\, a_{\,2} \,,\, \cdots \,,\, a_{\,n}\, \right >\,f_{\,i}.\]
\end{theorem}

\begin{proof}
Since \,$\left\{\,f_{\,i}\,\right\}^{\infty}_{i \,=\, 1}$\, is a tight frame associated to \,$\left(\,a_{\,2},\, \cdots,\, a_{\,n}\,\right)$\, with bound \,$A$,
\[\sum\limits^{\infty}_{i \,=\, 1}\,|\,\left<\,f \,,\, f_{\,i} \,|\, a_{\,2} \,,\, \cdots \,,\, a_{\,n}\,\right>\,|^{\,2} \,=\, A\,\left\|\,f \,,\, a_{\,2} \,,\, \cdots \,,\, a_{\,n}\,\right\|^{\,2}\; \;\forall\; f \,\in\, X_{F}.\]   
Let \,$S_{F}$\, be the corresponding frame operator for \,$\left\{\,f_{\,i}\,\right\}^{\infty}_{i \,=\, 1}$, then by (\ref{eq1.1}),
\[\left <\,S_{F}\,f \,,\, f \,|\, a_{\,2} \,,\, \cdots \,,\, a_{\,n} \,\right > \,=\, \sum\limits^{\infty}_{k \,=\, 1}\,\left|\,\left<\,f \,,\, f_{\,i} \,|\, a_{\,2} \,,\, \cdots \,,\, a_{\,n} \,\right> \,\right|^{\,2}\hspace{3cm}\] 
\[ \hspace{4cm}\,=\, A\, \left\|\,f \,,\, a_{\,2} \,,\, \cdots \,,\, a_{\,n}\,\right\|^{\,2} \,=\, \left <\,A\,f \,,\, f \,|\, a_{\,2} \,,\, \cdots \,,\, a_{\,n}\,\right>\]
\[ \Rightarrow\, \left <\,\left(\,S_{F} \,-\, A\, I_{F}\,\right) f \,,\, f \,|\, a_{\,2} \,,\, \cdots \,,\, a_{\,n}\,\right > \,=\, 0\; \;\;\forall\; f \,\in\, X_{F}\; \Rightarrow\, S_{F} \,=\, A\, I_{F}.\hspace{1cm}\] Therefore, for \,$f \,\in\, X_{F}$, we get
\[ A\,f \,=\, S_{F}\,(\,f\,) \,=\, \sum\limits^{\infty}_{i \,=\, 1}\; \left<\,f \,,\, f_{\,i} \,|\, a_{\,2} \,,\, \cdots \,,\, a_{\,n} \,\right>\,f_{\,i}\]
\[\Rightarrow\, f \,=\, \dfrac{1}{A}\, \sum\limits^{\infty}_{i \,=\, 1}\, \left<\,f \,,\, f_{\,i} \,|\, a_{\,2} \,,\, \cdots \,,\, a_{\,n} \,\right>\,f_{\,i}.\] 
\end{proof}

\begin{theorem}
Let \,$\{\,f_{\,i}\,\}^{\infty}_{i \,=\, 1}$\; be a frame associated to \,$\left(\,a_{\,2},\, \cdots,\, a_{\,n}\,\right)$\, for \,$X$\, and \,$S_{F}$\, be the corresponding frame operator.\;Then \,$\left \{\,S_{F}^{\,-\,\dfrac{1}{2}}\,f_{\,i}\,\right \}^{\infty}_{i \,=\, 1}$\; is a normalized tight frame associated to \,$\left(\,a_{\,2},\, \cdots,\, a_{\,n}\,\right)$\; and furthermore, for each \,$f \,\in\, X_{F}$, 
\[ f \,=\, \sum\limits^{\infty}_{i \,=\, 1}\,\left <\,f \;,\; S_{F}^{\,-\,\dfrac{1}{2}}\;f_{\,i} \;|\;  a_{\,2} \;,\; \cdots \;,\; a_{\,n}\;\right >\;S_{F}^{\,-\,\dfrac{1}{2}}\;f_{\,i}.\]
\end{theorem}

\begin{proof}
By Theorem (\ref{th1.05}), a unique positive square root \,$S^{\,-\, \dfrac{1}{2}}_{F}$\, of \,$S_{F}^{\,-1}$\; exists, which is self-adjoint and commutes with \,$S_{F}$.\;Therefore, each \,$f \,\in\, X_{F}$\, can be written as
\[ f \;=\; S_{F}^{\,-\,\dfrac{1}{2}}\; S_{F}^{\,-\,\dfrac{1}{2}}\; \left(\,S_{F}\;f\,\right) \;=\; S_{F}^{\,-\,\dfrac{1}{2}}\; S_{F}\;\left(\,S_{F}^{\,-\,\dfrac{1}{2}}\;f \,\right)\]
\[\hspace{1cm}=\; S_{F}^{\,-\,\dfrac{1}{2}}\; \left(\,\sum\limits^{\infty}_{i \,=\, 1}\,\left <\,S_{F}^{\,-\,\dfrac{1}{2}}\,f \;,\; f_{\,i} \,|\, a_{\,2} \,,\, \cdots \,,\, a_{\,n}\;\right >\,f_{\,i}\,\right)\]
\[=\; \sum\limits^{\infty}_{i \,=\, 1}\,\left <\,S_{F}^{\,-\,\dfrac{1}{2}}\,f \,,\, f_{\,i} \,|\, a_{\,2} \,,\, \cdots \,,\, a_{\,n}\,\right >\,S_{F}^{\,-\,\dfrac{1}{2}}\;f_{\,i}\] 
\[=\; \sum\limits^{\infty}_{i \,=\, 1}\,\left <\,f \,,\, S_{F}^{\,-\,\dfrac{1}{2}}\,f_{\,i} \,|\, a_{\,2} \,,\, \cdots \,,\, a_{\,n}\,\right >\,S_{F}^{\,-\,\dfrac{1}{2}}\,f_{\,i}.\]
Now, for each \,$f \,\in\, X_{F}$, we have
\[\|\,f \,,\, a_{\,2} \,,\, \cdots \,,\, a_{\,n} \,\|^{\,2} \,=\, \left<\,f \,,\, f \,|\, a_{\,2} \,,\, \cdots \,,\, a_{\,n}\,\right>\]
\[=\, \left<\,\sum\limits^{\infty}_{i \,=\, 1}\,\left <\,f \,,\, S_{F}^{\,-\,\dfrac{1}{2}}\,f_{\,i} \;|\; a_{\,2} \,,\, \cdots \,,\, a_{\,n}\,\right >\,S_{F}^{\,-\,\dfrac{1}{2}}\,f_{\,i} \,,\, f \,|\, a_{\,2} \,,\, \cdots \,,\, a_{\,n}\,\right>\]
\[=\,\sum\limits^{\infty}_{i \,=\, 1}\,\left <\,f \,,\, S_{F}^{\,-\,\dfrac{1}{2}}\,f_{\,i} \;|\; a_{\,2} \,,\, \cdots \,,\, a_{\,n}\,\right >\,\left<\,S_{F}^{\,-\,\dfrac{1}{2}}\,f_{\,i} \,,\, f \,|\, a_{\,2} \,,\, \cdots \,,\, a_{\,n}\,\right> \]
\[ \,=\, \sum\limits^{\infty}_{i \,=\, 1}\, \left|\,\left <\,f \,,\, S_{F}^{\,-\,\dfrac{1}{2}}\,f_{\,i} \,|\, a_{\,2} \,,\, \cdots \,,\, a_{\,n}\,\right >\,\right|^{\,2}.\hspace{4.1cm}\]
Hence, \,$\left \{\,S_{F}^{\,-\,\dfrac{1}{2}}\,f_{\,i}\,\right \}^{\infty}_{i \,=\, 1}$\, is a normalized tight frame associated to \,$\left(\,a_{\,2},\, \cdots,\, a_{\,n}\,\right)$.
\end{proof}

\section{Frame and operator relative to $n$-Hilbert space}

\smallskip\hspace{.6 cm}In this section, we establish an image of frame associated to \,$\left(\,a_{\,2},\, \cdots,\, a_{\,n}\,\right)$\, under a bounded linear operator becomes a frame associated to \,$\left(\,a_{\,2},\, \cdots,\, a_{\,n}\,\right)$\, if and only if the bounded linear operator have to be invertible.\;In general, a Bessel sequence does not a frame in \,$n$-Hilbert space.\;We give some sufficient condition for Bessel sequence associated to \,$\left(\,a_{\,2},\, \cdots,\, a_{\,n}\,\right)$\, becomes frame associated to \,$\left(\,a_{\,2},\, \cdots,\, a_{\,n}\,\right)$\, in \,$n$-Hilbert space.

\begin{theorem}\label{th3}
Let \,$\{\,f_{\,i}\,\}_{i \,=\,1}^{\infty}$\; be a frame associated to \,$\left(\,a_{\,2},\, \cdots,\, a_{\,n}\,\right)$\; for \,$X$\, with bounds \,$A,\, B$\, and \,$S_{F}$\; be the corresponding frame operator and \,$U \,:\, X_{F} \,\to\, X_{F}$\; be a bounded linear operator.\;Then \,$\left\{\, U\,f_{\,i}\,\right\}_{i \,=\,1}^{\infty}$\, is a frame associated to \,$\left(\,a_{\,2},\, \cdots,\, a_{\,n}\,\right)$\, for \,$X$\; if and only if \,$U$\; is invertible on \,$X_{F}$.
\end{theorem}

\begin{proof}
Suppose \,$U \,:\, X_{F} \,\to\, X_{F}$\; is invertible.\;Then for each \,$f \,\in\, X_{F}$,
\[\left\|\,f \,,\, a_{\,2} \,,\, \cdots \,,\, a_{\,n}\,\right\|^{\,2} \,=\, \left\|\,(\,U^{\,-\, 1}\,)^{\,\ast}\,U^{\,\ast}\,f \,,\, a_{\,2} \,,\, \cdots \,,\, a_{\,n}\,\right\|^{\,2}\hspace{1cm}\]
\begin{equation}\label{eq3}
\hspace{2.8cm}\leq\, \left\|\,U^{\,-\, 1}\,\right \|^{\,2}\,\left\|\,U^{\,\ast}\,f \,,\, a_{\,2} \,,\, \cdots \,,\, a_{\,n}\,\right \|^{\,2}.
\end{equation} 
Since \,$\{\,f_{\,i}\,\}^{\infty}_{i \,=1\, }\; \text{is a frame associated to}\;\left(\,a_{\,2},\, \cdots,\, a_{\,n}\,\right)$, for each \,$f \,\in\, X_{F}$, 
\[\sum\limits^{\infty}_{i \,=\, 1}\, \left |\,\left <\,f \,,\, U\,f_{\,i} \,|\, a_{\,2} \,,\, \cdots \,,\, a_{\,n}\,\right >\,\right |^{\,2} \,=\, \sum\limits^{\infty}_{i \,=\, 1}\, \left|\,\left <\,U^{\,\ast}\,f \,,\, f_{\,i} \,|\, a_{\,2} \,,\, \cdots \,,\, a_{\,n}\,\right >\,\right |^{\,2}\hspace{3cm}\]
\[\hspace{1.4cm}\geq\, A\, \left\|\,U^{\,\ast}\,f \,,\, a_{\,2} \,,\, \cdots \,,\, a_{\,n}\,\right \|^{\,2}\]
\[\hspace{4.5cm}\,\geq\, A\, \left\|\,U^{\,-\, 1}\,\right \|^{\,-\, 2}\, \left\|\,f \,,\, a_{\,2} \,,\, \cdots \,,\, a_{\,n}\,\right\|^{\,2}\; \;[\;\text{by (\ref{eq3})}\;].\]
On the other hand,
\[\sum\limits^{\infty}_{i \,=\, 1}\, \left |\,\left <\,f \,,\, U\,f_{\,i} \,|\, a_{\,2} \,,\, \cdots \,,\, a_{\,n}\,\right >\,\right |^{\,2} \,=\, \sum\limits^{\infty}_{i \,=\, 1}\, \left|\,\left <\,U^{\,\ast}\,f \,,\, f_{\,i} \,|\, a_{\,2} \,,\, \cdots \,,\, a_{\,n}\,\right >\,\right |^{\,2}\]
\[\hspace{5cm}\,\leq\, B\, \left\|\,U\,\right \|^{\,2}\, \left\|\,f \,,\, a_{\,2} \,,\, \cdots \,,\, a_{\,n}\,\right\|^{\,2}. \]
Hence, \,$\{\, U\,f_{\,i}\,\}_{i \,=\,1}^{\infty}$\; is a frame associated to \,$\left(\,a_{\,2},\, \cdots,\, a_{\,n}\,\right)$\, for \,$X$.\\

Conversely, suppose \,$\{\, U\,f_{\,i}\,\}_{i \,=\,1}^{\infty}$\; is a frame associated to \,$\left(\,a_{\,2},\, \cdots,\, a_{\,n}\,\right)$\, for \,$X$.\;Now, for each \,$f \,\in\, X_{F}$, we have
\[\sum\limits^{\infty}_{i \,=\, 1}\, \left <\,f,\, U\,f_{\,i} \,|\, a_{\,2},\, \cdots,\, a_{\,n}\,\right >\,U\,f_{\,i} \,=\, U\,\left(\,\sum\limits^{\infty}_{i \,=\, 1}\, \left <\,U^{\,\ast}\,f,\, f_{\,i} \,|\, a_{\,2},\, \cdots,\, a_{\,n}\,\right >\,\right) \,=\, U\,S_{F}\, U^{\,\ast}\,f.\]
This implies that \,$U\,S_{F}\, U^{\,\ast}$\; is the corresponding frame operator for \,$\{\, U\,f_{\,i}\,\}_{i \,=\,1}^{\infty}$.\;By Theorem (\ref{th2.2}), \,$U\,S_{F}\, U^{\,\ast}$\; is invertible and hence \,$U \,:\, X_{F} \,\to\, X_{F}$\; is invertible.
\end{proof}

\begin{theorem}
Let \,$\{\,f_{\,i}\,\}_{i \,=\,1}^{\infty}$\, be a frame associated to \,$\left(\,a_{\,2},\, \cdots,\, a_{\,n}\,\right)$\, and \,$U \,:\, X_{F} \,\to\, X_{F}$\, be a bounded linear operator.\;Then \,$\left\{\, f_{\,i} \,+\, U\,f_{\,i}\,\right\}_{i \,=\,1}^{\infty}$\, is a frame associated to \,$\left(\,a_{\,2},\, \cdots,\, a_{\,n}\,\right)$\, for \,$X$\, if and only if \,$I \,+\, U$\, is invertible on \,$X_{F}$.
\end{theorem}

\begin{proof}
For each \,$f \,\in\, X_{F}$, we can write
\[ \sum\limits^{\infty}_{i \,=\, 1}\, \left|\,\left <\,f \,,\, f_{\,i} \,+\, U\, f_{\,i} \,|\, a_{\,2} \,,\, \cdots \,,\, a_{\,n}\,\right >\,\right|^{\,2} \,=\, \sum\limits^{\infty}_{i \,=\, 1}\, \left|\,\left <\,\left(\,I \,+\, U\,\right)^{\,\ast}\,f \,,\, f_{\,i} \,|\, a_{\,2} \,,\, \cdots \,,\, a_{\,n}\,\right >\,\right |^{\,2}.\]
Thus, \,$\left\{\,f_{\,i} \,+\, U\,f_{\,i}\,\right\}_{i \,=\,1}^{\infty}$\, is a frame associated to \,$\left(\,a_{\,2},\, \cdots,\, a_{\,n}\,\right)$\, if and only if \,$\{\,f_{\,i}\,\}_{i \,=\,1}^{\infty}$\; is a frame associated to \,$\left(\,a_{\,2},\, \cdots,\, a_{\,n}\,\right)$.\;By Theorem (\ref{th3}), \,$\left\{\,f_{\,i} \,+\, U\,f_{\,i}\,\right\}_{i \,=\,1}^{\infty}$\; is a frame associated to \,$\left(\,a_{\,2},\, \cdots,\, a_{\,n}\,\right)$\, if and only if \,$I \,+\, U$\; is invertible on \,$X_{F}$.
\end{proof}

\begin{remark}
Furthermore, for each \,$f \,\in\, X_{F}$, we have
\[\sum\limits^{\infty}_{i \,=\, 1}\, \left<\,f \,,\, f_{\,i} \,+\, U\,f_{\,i} \,|\, a_{\,2} \,,\, \cdots \,,\, a_{\,n}\,\right >\, \left(\,f_{\,i} \,+\, U\, f_{\,i}\,\right)\hspace{2.5cm}\]
\[\hspace{.6cm}\,=\, \sum\limits^{\infty}_{i \,=\, 1}\, \left <\,f \,,\, (\,I \,+\, U\,)\,f_{\,i} \,|\, a_{\,2} \,,\, \cdots \,,\, a_{\,n}\,\right >\, (\,I \,+\, U \,)\,f_{\,i}\] 
\[\hspace{.8cm}\,=\, (\,I \,+\, U\,)\, \sum\limits^{\infty}_{i \,=\, 1}\, \left<\,(\,I \,+\, U\,)^{\,\ast}\, f \,,\, f_{\,i} \,|\, a_{\,2} \,,\, \cdots \,,\, a_{\,n}\,\right >\,f_{\,i}\]
\[=\, (\,I \,+\, U\,)\,S_{F}\,(\,I \,+\, U \,)^{\,\ast}\,f. \hspace{3cm}\]
This implies that \,$(\,I \,+\, U\,)\,S_{F}\,(\,I \,+\, U \,)^{\,\ast}$\, is the corresponding frame operator for the frame \,$\left\{\, f_{\,i} \,+\, U\,f_{\,i}\,\right\}_{i \,=\,1}^{\infty}$. 
\end{remark}

\begin{theorem}
Let \,$\{\,f_{\,i}\,\}_{i \,=\,1}^{\infty}$\, and \,$\{\,g_{\,i}\,\}_{i \,=\,1}^{\infty}$\; be two Bessel sequences associated to \,$\left(\,a_{\,2},\, \cdots,\, a_{\,n}\,\right)$\, for \,$X$\, with pre frame operators \,$T_{F}$, \,$T_{F}^{\,\prime}$, respectively.\;Then for \,$L_{\,1},\, L_{\,2} \,\in\, \mathcal{B}\,\left(\,X_{F}\,\right)$, the sequence \,$\left\{\,L_{\,1}\,f_{\,i}\, \,+\, L_{\,2}\,g_{\,i}\,\right\}_{i \,=\,1}^{\infty}$\; is a frame associated to \,$\left(\,a_{\,2},\, \cdots,\, a_{\,n}\,\right)$\, for \,$X$\; if and only if \,$\left[\;T_{F}^{\,\ast}\, L_{\,1}^{\,\ast} \,+\, \left(\,T_{F}^{\,\prime}\,\right)^{\,\ast}\, L_{\,2}^{\,\ast}\, \right]$\; is an invertible on \,$X_{F}$.
\end{theorem}

\begin{proof}
Since \,$T_{F},\, T_{F}^{\,\prime}$\, are pre frame operators for \,$\{\,f_{\,i}\,\}_{i \,=\,1}^{\infty} \;\text{and}\; \{\,g_{\,i}\,\}_{i \,=\,1}^{\infty}$,  
\[T_{F}^{\,\ast}\,(\,f\,) \,=\, \left\{\,\left <\,f \,,\, f_{\,i}\, \,|\, a_{\,2} \,,\, \cdots \,,\, a_{\,n}\,\right>\,\right\}_{i \,=\,1}^{\infty}, \;\text{and}\] 
\[\hspace{1cm}\left(\,T_{F}^{\,\prime}\,\right)^{\,\ast}\,(\,f\,) \,=\, \left\{\,\left <\,f \,,\, g_{\,i}\, \,|\, a_{\,2} \,,\, \cdots \,,\, a_{\,n}\,\right>\,\right\}_{i \,=\,1}^{\infty}\; \;\forall\; f \,\in\, X_{F}.\]
By Theorem (\ref{th3}), \,$\left\{\,L_{\,1}\,f_{\,i}\, \,+\, L_{\,2}\,g_{\,i}\,\right\}_{i \,=\,1}^{\infty}$\, is a frame associated to \,$\left(\,a_{\,2},\, \cdots,\, a_{\,n}\,\right)$\, for \,$X$\; if and only if its analysis operator \,$T \,:\, X_{F} \,\to\, l^{\,2}\,(\,\mathbb{N}\,)$\; defined by
\[T\,(\,f\,) \,=\, \left \{\,\left <\,f \,,\, L_{\,1}\,f_{\,i}\, \,+\, L_{\,2}\,g_{\,i} \,|\, a_{\,2} \,,\, \cdots \,,\, a_{\,n}\,\right>\,\right \}_{i \,=\,1}^{\infty}\]\; is invertible on \,$X_{F}$.\;Now, for each \,$f \,\in\, X_{F}$,
\[T\,(\,f\,) \,=\, \left \{\,\left <\,f \,,\, L_{\,1}\,f_{\,i}\, \,+\, L_{\,2}\,g_{\,i} \,|\, a_{\,2} \,,\, \cdots \,,\, a_{\,n}\,\right>\,\right \}_{i \,=\,1}^{\infty}\hspace{3cm}\] 
\[\hspace{1.3cm}\,=\, \left\{\,\left <\,f \,,\, L_{\,1}\,f_{\,i}\, \,|\, a_{\,2} \,,\, \cdots \,,\, a_{\,n}\,\right> \,+\, \left <\,f \,,\, L_{\,2}\,g_{\,i}\, \,|\, a_{\,2} \,,\, \cdots \,,\, a_{\,n}\,\right>\right \}_{i \,=\,1}^{\infty}\]
\[\hspace{2.2cm} \,=\, \left\{\,\left <\,L_{\,1}^{\,\ast}\,f \,,\, f_{\,i} \,|\, a_{\,2} \,,\, \cdots \,,\, a_{\,n}\,\right>\,\right\}_{i \,=\,1}^{\infty} \,+\,  \left\{\,\left <\,L_{\,2}^{\,\ast}\,f \,,\, g_{\,i} \,|\, a_{\,2} \,,\, \cdots \,,\, a_{\,n}\,\right>\,\right\}_{i \,=\,1}^{\infty} \]
\[\hspace{1cm}\,=\, \left[\,T_{F}^{\,\ast}\,L_{\,1}^{\,\ast} \,+\, \left(\,T_{F}^{\,\prime}\,\right)^{\,\ast}\,L_{\,2}^{\,\ast}\,\;\right]\,f.\hspace{5.2cm}\]
Therefore, \,$\left\{\,L_{\,1}\,f_{\,i}\, \,+\, L_{\,2}\,g_{\,i}\,\right\}_{i \,=\,1}^{\infty}$\; is a frame associated to \,$\left(\,a_{\,2},\, \cdots,\, a_{\,n}\,\right)$\, for \,$X$\; if and only if \,$\left[\,T_{F}^{\,\ast}\,L_{\,1}^{\,\ast} \,+\, \left(\,T_{F}^{\,\prime}\,\right)^{\,\ast}\,L_{\,2}^{\,\ast}\,\;\right]$\; is an invertible on \,$X_{F}$.
\end{proof}

\begin{theorem}
A sequence \,$\left\{\,f_{\,i}\,\right\}^{\infty}_{i \,=\, 1}$\; in \,$X$\; is a frame associated to \,$\left(\,a_{\,2},\, \cdots,\, a_{\,n}\,\right)$\; if and only if \,$T_{F} \,:\, \left\{\,c_{i}\,\right\}_{i \,=\, 1}^{\,\infty} \,\to\, \sum\limits^{\infty}_{i \,=\, 1}\,c_{\,i}\,f_{\,i}$\; is well-defined mapping of \,$l^{\,2}\,(\,\mathbb{N}\,)$\, onto \,$X_{F}$. 
\end{theorem}

\begin{proof}
First we suppose that \,$\left\{\,f_{\,i}\,\right\}^{\infty}_{i \,=\, 1}$\; is a frame associated to \,$\left(\,a_{\,2},\, \cdots,\, a_{\,n}\,\right)$. Then by Theorem (\ref{th2.1}), \,$T_{F}$\; is well-defined bounded linear operator from \,$l^{\,2}\,(\,\mathbb{N}\,)$\; onto \,$X_{F}$.\;Also by Theorem (\ref{th2.2}), the frame operator \,$S_{F} \,=\, T_{F}\, T^{\,\ast}_{F}$\; is surjective and hence \,$T_{F}$\; is surjective.\\

\text{Conversely}, suppose that \,$T_{F}$\; is well-defined mapping of \,$l^{\,2}\,(\,\mathbb{N}\,)$\; onto \,$X_{F}$.\;By Theorem (\ref{th1.1}), \,$T_{F}$\; is bounded and that \,$\{\,f_{\,i}\,\}^{\infty}_{i \,=\, 1}$\; is a Bessel sequence associated to \,$\left(\,a_{\,2},\, \cdots,\, a_{\,n}\,\right)$.\;So, \,$T_{F}^{\,\ast}\,f \,=\, \left\{\,\left<\,f \,,\, f_{\,i} \,|\, a_{\,2} \,,\, \cdots \,,\, a_{\,n} \,\right>\,\right\}_{i \,=\, 1}^{\,\infty}$.\;Since \,$T_{F}$\, is surjective, by Theorem (\ref{th1}), there exists an operator \,$T^{\,\dagger}_{F} \,:\, X_{F} \,\to\, l^{\,2}\,(\,\mathbb{N}\,)$\; such that \,$T_{F}\; T^{\,\dagger}_{F} \,=\, I_{F}$.\;This implies that \,$\left(\,T^{\,\dagger}_{F}\,\right)^{\,\ast}\,T^{\,\ast}_{F} \,=\, I_{F}$.\;Then for each \,$f \,\in\, X_{F}$, we get
\[\left\|\,f \,,\, a_{\,2} \,,\, \cdots \,,\, a_{\,n}\,\right\|^{\,2} \,\leq\, \left\|\,\left(\,T^{\,\dagger}_{F}\,\right)^{\,\ast}\,\right\|^{\,2}\; \left\|\,T^{\,\ast}_{F}\,f \,,\, a_{\,2} \,,\, \cdots \,,\, a_{\,n}\,\right\|^{\,2}\]
\[\hspace{3.8cm} \leq\, \left\|\, T^{\,\dagger}_{F}\,\right\|^{\,2}\, \sum\limits^{\infty}_{i \,=\, 1}\,|\,\left<\,f \,,\, f_{\,i} \,|\, a_{\,2} \,,\, \cdots \,,\, a_{\,n}\,\right>\,|^{\,2}\]
\[\Rightarrow\, \dfrac{1}{\left\|\, T^{\,\dagger}_{F}\,\right\|^{\,2}}\, \left\|\,f \,,\, a_{\,2} \,,\, \cdots \,,\, a_{\,n}\,\right\|^{\,2} \,\leq\, \sum\limits^{\infty}_{i \,=\, 1}\,|\,\left<\,f \,,\, f_{\,i} \,|\, a_{\,2} \,,\, \cdots \,,\, a_{\,n}\,\right>\,|^{\,2}.\]
Therefore \,$\{\,f_{\,i}\,\}^{\infty}_{i \,=\, 1}$\; is a frame associated to \,$\left(\,a_{\,2},\; \cdots,\; a_{\,n}\,\right)$.        
\end{proof}

\begin{theorem}
Let \,$\left\{\,f_{\,i}\,\right\}^{\infty}_{i \,=\, 1}$\; and \,$\left\{\,g_{\,i}\,\right\}^{\infty}_{i \,=\, 1}$\; be two Bessel sequences associated to \,$\left(\,a_{\,2},\, \cdots,\, a_{\,n}\,\right)$\, for \,$X$\, with bounds \,$C\; \;\text{and}\; \;D$, respectively.\;Suppose that \,$T_{F}$\; and \,$T_{F}^{\,\prime}$\; be their pre frame operators such that \,$T_{F}\,\left(\,T_{F}^{\,\prime}\,\right)^{\,\ast} \,=\, I_{F}$.\;Then \,$\left\{\,f_{\,i}\,\right\}^{\infty}_{i \,=\, 1}$\; and \,$\left\{\,g_{\,i}\,\right\}^{\infty}_{i \,=\, 1}$\; are frames associated to \,$\left(\,a_{\,2},\, \cdots,\, a_{\,n}\,\right)$\; for \,$X$. 
\end{theorem}

\begin{proof}
Since \,$T_{F}$\, and \,$T_{F}^{\,\prime}$\, are pre frame operators for \,$\left\{\,f_{\,i}\,\right\}^{\infty}_{i \,=\, 1}$\, and \,$\left\{\,g_{\,i}\,\right\}^{\infty}_{i \,=\, 1}$, respectively, for each \,$f \,\in\, X_{F}$, we have 
\[\left\|\,T_{F}^{\,\ast}\,f\,\right\|_{F}^{\,2} \,=\, \sum\limits^{\infty}_{i \,=\, 1}\,\left |\, \left <\,f \,,\, f_{\,i} \,|\, a_{\,2} \,,\, \cdots \,,\, a_{\,n}\,\right >\,\right |^{\,2}, \;\text{and}\]
\[\left\|\,\left(\,T_{F}^{\,\prime}\,\right)^{\,\ast}\,f\,\right\|_{F}^{\,2} \,=\, \sum\limits^{\infty}_{i \,=\, 1}\,\left |\, \left <\,f \,,\, g_{\,i} \,|\, a_{\,2} \,,\, \cdots \,,\, a_{\,n}\,\right >\,\right |^{\,2}.\] Now, for each \,$f \,\in\, X_{F}$, we have 
\[\|\,f \,,\, a_{\,2} \,,\, \cdots \,,\, a_{\,n}\,\|^{\,4} \,=\, \left[\,\left<\,f \,,\, f \,|\, a_{\,2} \,,\, \cdots \,,\, a_{\,n}\,\right>\,\right]^{\,2}\hspace{3cm}\]
\[\hspace{2.2cm}\,=\, \left[\,\left<\,T_{F}\,\left(\,T_{F}^{\,\prime}\,\right)^{\,\ast}\,f \,,\, f \,|\, a_{\,2} \,,\, \cdots \,,\, a_{\,n}\,\right>\,\right]^{\,2}\; \;[\;\because\; T_{F}\,\left(\,T_{F}^{\,\prime}\,\right)^{\,\ast} \,=\, I_{F}\;]\]
\[\,=\, \left[\,\left<\,\left(\,T_{F}^{\,\prime}\,\right)^{\,\ast}\,f \,,\, T_{F}^{\,\ast}\,f \,|\, a_{\,2} \,,\, \cdots \,,\, a_{\,n}\,\right>\,\right]^{\,2}\hspace{2cm}\]
\[\hspace{1.8cm} \,\leq\, \left\|\,T_{F}^{\,\ast}\,f\,\right\|_{F}^{\,2}\; \left\|\,\left(\,T_{F}^{\,\prime}\,\right)^{\,\ast}\,f\,\right\|_{F}^{\,2}\; \;[\;\text{by Cauchy-Schwarz inequality}\;]\]
\[\hspace{2cm}=\, \sum\limits^{\infty}_{i \,=\, 1}\,\left |\, \left <\,f \,,\, f_{\,i} \,|\, a_{\,2} \,,\, \cdots \,,\, a_{\,n}\,\right >\,\right |^{\,2}\; \sum\limits^{\infty}_{i \,=\, 1}\,\left |\, \left <\,f \,,\, g_{\,i} \,|\, a_{\,2} \,,\, \cdots \,,\, a_{\,n}\,\right >\,\right |^{\,2}\] 
\[\leq\, \sum\limits^{\infty}_{i \,=\, 1}\,\left |\, \left <\,f \,,\, f_{\,i} \,|\, a_{\,2} \,,\, \cdots \,,\, a_{\,n}\,\right >\,\right |^{\,2}\; D\; \|\,f \,,\,  a_{\,2} \,,\, \cdots \,,\, a_{\,n}\,\|^{\,2}\]
\[\Rightarrow\; \dfrac{1}{D}\; \|\,f \,,\,  a_{\,2} \,,\, \cdots \,,\, a_{\,n}\,\|^{\,2} \;\leq\;  \sum\limits^{\infty}_{i \,=\, 1}\,\left |\, \left <\,f \,,\, f_{\,i} \,|\, a_{\,2} \,,\, \cdots \,,\, a_{\,n}\,\right >\,\right |^{\,2}.\]Hence, \,$\left\{\,f_{\,i}\,\right\}^{\infty}_{i \,=\, 1}$\; is a frame associated to \,$\left(\,a_{\,2},\, \cdots,\, a_{\,n}\,\right)$\, in \,$X$.\;Similarly, it can be shown that \,$\left\{\,g_{\,i}\,\right\}^{\infty}_{i \,=\, 1}$\; is a frames associated to \,$\left(\,a_{\,2},\, \cdots,\, a_{\,n}\,\right)$\; with the lower bound \,$\dfrac{1}{C}$.     
\end{proof}

\end{document}